\documentclass[submission%
]{dmtcs}

\usepackage[latin1]{inputenc}
\usepackage{subfigure,comment}

\usepackage{epic,tikz,amsmath,graphicx}
\newtheorem{thm}{Theorem}

\newtheorem{prop}[thm]{Proposition}
\newtheorem{defn}[thm]{Definition}

\newtheorem{cor}[thm]{Corollary}
\newtheorem{lem}[thm]{Lemma}

\def\emm#1,{{\em #1}}
\newcommand{\zs}{\mathbb{Z}}%{\mbox{\bbold Z}}

\newcommand{\qs}{\mathbb{Q}}%{\mbox{\bbold Q}}

\newcommand{\beq}{\begin{equation}}
\newcommand{\eeq}{\end{equation}}
\newcommand{\gf}{generating function}

\newcommand{\bx}{1/x}
\newcommand{\by}{1/y}

\newcommand{\cS}{\mathcal S}
\newcommand{\om}{\omega}

\usepackage{amssymb}

\author{O. Bernardi\addressmark{1},
M. Bousquet-M\'elou\addressmark{2}
\and K. Raschel\addressmark{3}
}
\title[Counting  quadrant walks via Tutte's invariant method]
{Counting  quadrant walks via Tutte's invariant method (extended abstract)}
\address{\addressmark{1} Brandeis University, Department of Mathematics, 415 South Street, Waltham, MA 02453, US\\
  \addressmark{2} CNRS, LaBRI, Universit\'e de Bordeaux, 351 cours de la Lib\'eration,  33405 Talence Cedex, France\\
  \addressmark{3}
CNRS, LMPT, Universit\'e de Tours, Parc de Grandmont, 37200 Tours, France}
\received{\today}
\revised{TBA}
\accepted{TBA}
\begin{document}
\maketitle
\begin{abstract}
\paragraph{Abstract.} 
In the 1970s, Tutte developed a clever algebraic
approach, based on certain ``invariants'', to solve a functional
equation that arises in the enumeration of properly colored
%planar
 triangulations.  The enumeration of plane lattice walks
confined to the first quadrant is governed by similar equations, and
has led in the past decade to a rich collection of attractive results
dealing with the nature (algebraic, D-finite or not) of the associated
generating function, depending on the set of allowed steps.
%, at least when the walk take small steps, that is, steps of $\{-1, 0,1\}^2$. 

We first adapt Tutte's approach to prove (or reprove) the algebraicity
of all quadrant models known or conjectured to be
algebraic (with one small exception). This includes Gessel's famous
model, and the first proof ever found for one  model with weighted
steps. To be applicable, the method requires the existence of two rational functions called   \emph{invariant} and  \emph{decoupling function} respectively. When they exist, algebraicity comes
out (almost) automatically.

Then, we move to an analytic viewpoint which has already proved very
powerful, leading in particular to integral expressions of the
generating function in the non-D-finite cases, as well as to proofs of
non-D-finiteness. We develop  in this
context a weaker notion of  invariant. Now all quadrant models have invariants, and for those that
have in addition a decoupling function, we obtain integral-free
expressions of the generating function, and a proof that this series is differentially
algebraic (that is, satisfies a non-linear differential equation).
This analytic approach solves as well the algebraic model left unsolved in the first
part.

\paragraph{R\'esum\'e.} Nous adaptons à l'énumération des chemins
confinés dans le
premier quadrant une méthode développée par Tutte dans les années 70 pour
compter les triangulations proprement colorées.

Nous prouvons ou reprouvons d'abord ainsi l'algébricité de tous les
chemins du quadrant dont la série génératrice est  (ou est conjecturée)
algébrique ; ceci inclut le célèbre modèle de Gessel. Pour \^etre
applicable, la méthode requiert l'existence de deux fonctions
rationnelles appelées \emph{invariant} et \emph{fonction de découplage}.

Nous passons
ensuite à un cadre analytique qui a déjà fourni des
expressions intégrales des séries génératrices dans des cas non
D-finis, et des preuves de non-D-finitude. Dans ce
contexte, nous définissons une notion plus faible d'invariants. 
Tous les modèles du quadrant admettent alors un invariant, et, pour ceux qui ont de plus
une fonction de découplage, nous obtenons une expression sans
intégrale de la série génératrice, et prouvons que cette série est
différentiellement algébrique (solution d'une équation différentielle polynomiale).
\end{abstract}

%%%%%%%%%%%%%%%%%%%%%%%%%%%%%%%%%%%%%%%%%%%%%%%%%%%%%%%%%%%%%%%%%
\section{Introduction}
%%%%%%%%%%%%%%%%%%%%%%%%%%%%%%%%%%%%%%%%%%%%%%%%%%%%%%%%%%%%%%%%%
In the past decade, the enumeration of plane walks confined to the
first quadrant has received a lot of attention, and given rise to many
interesting methods and results. Given a set of steps $\cS \subset
\zs^2$ and a starting
point (usually $(0,0)$), the main question is to determine the
\gf
\begin{equation*}
     Q(x,y;t) \equiv Q(x,y)=\sum_{i,j,n\geq 0} q(i,j;n) x^iy^jt^n,
\end{equation*}
where $q(i,j;n)$ is the number of $n$-step quadrant walks from $(0,0)$
to $(i,j)$, taking their steps in $\cS$. If one only considers walks with \emm small
steps, (that is, $\cS\subset \{-1,0,1\}^2$), there are 79 inherently
different step sets (also called \emm models), and an expression of
$Q(x,y)$, sometimes rather complex, is known in each case. Moreover,
the \emm nature, of this series is also known: it is \emm D-finite, (that
is, satisfies three linear differential equations, one in $x$, one in
$y$, one in $t$, with polynomial coefficients) if and only if a
certain group of rational transformations is finite. This happens in
23 cases. In exactly 4 of them, $Q(x,y)$ is even \emm algebraic,, 
that is, satisfies a polynomial equation with polynomial coefficients in $x$,
$y$ and $t$. 
This classification has been obtained by  an attractive combination of
approaches:  algebraic~\cite{bousquet-versailles,BMM-10,gessel-proba,gessel-zeilberger},
computer-algebraic~\cite{BoKa-10,KaKoZe08,kauers07v},
analytic~\cite{BKR-13,Ra-12,KR-12},
asymptotic~\cite{denisov-wachtel,MiRe09}. 

The starting point of all of them is a functional equation
satisfied by $Q(x,y)$, which is simple to establish, but often hard
to solve. For instance, in the case of Kreweras' walks (steps
$\nearrow$, $\leftarrow$, $\downarrow$), it
reads
\beq\label{eq:Kreweras}
Q(x,y)=1+ t xy Q(x,y) + t\, \frac{Q(x,y)-Q(0,y)}x +  t \, \frac{Q(x,y)-Q(x,0)}{y}.
\eeq
This is reminiscent of an equation
written by Tutte in the 1970s when studying %properly
 $q$-colored triangulations:
\beq\label{eq:Tutte}
G(x,y)=xq(q-1)t^2+\frac{xy}{qt}G(1,y)G(x,y)-x^2yt\frac{G(x,y)-G(1,y)}{x-1}+x\frac{G(x,y)-G(x,0)}{y}.
\eeq
Due to the quadratic term, \eqref{eq:Tutte}
 is in fact more complicated
than~\eqref{eq:Kreweras}. Tutte worked about a decade on this
equation, and finally solved it, proving that the series $G(1,0)$ is
\emm differentially algebraic,, that is, satisfies a (non-linear)
differential equation in $t$. One key step in his study was to prove
that for certain (infinitely many) values of $q$, the series $G(x,y)$
is algebraic, using a certain 
notion of \emm invariant,~\cite{tutte-chromatic-revisited}.

\medskip
Could this notion bring something new to the classification of
quadrant walks? This paper answers this question positively:
\begin{itemize}
\item We first adapt Tutte's approach to quadrant walks, and thus
  obtain \emm short and uniform proofs, of algebraicity for all
  algebraic models (with one
  small exception). This includes the shortest proof ever found for
  Gessel's famous model, and extends to models with weighted steps,
  for which algebraicity was sometimes still conjectural~\cite{KaYa-15}. With our
  approach, a model with finite group is algebraic if and only if it
  admits a \emm decoupling function, (Sections~\ref{sec:gessel} and~\ref{sec:extensions}).
\item We then  define a weaker notion of invariant, and use it to
  give an integral free expression of $Q(x,y)$ for models with
  infinite group that admit a decoupling function
  (Section~\ref{sec:analysis};
see e.g.~\eqref{eq:Q(x,0)_Kreweras_model_6}).
 We  have at the   moment found 9 such models. This expression implies
  that $Q(x,y)$ is differentially algebraic in $x$, $y$ and $t$
  (Section~\ref{sec:DA}). 
\end{itemize}

We now introduce some basic tools  in the study of quadrant walks with
small steps.
 A simple step-by-step
construction of these  walks gives the following functional
equation~\cite{BMM-10}:  
\begin{equation}
\label{eq:functional_equation}
     K(x,y) Q(x,y) = K(x,0) Q(x,0) + K(0,y) Q(0,y)-K(0,0) Q(0,0)- xy,
\end{equation}
where 
$$%\begin{equation}\label{eq:kernel}
     K(x,y) = xy\Bigg(t\sum_{(i,j)\in \mathcal S} x^iy^j-1\Bigg)
$$% \end{equation}
is  the \emm kernel,  of the model.
It is a  polynomial of degree 2 in $x$ and
$y$,
% (small steps hypothesis), 
which  we often write  as
\beq\label{K-abc}
     K(x,y)  =\widetilde a(y)x^2+\widetilde
b(y)x+\widetilde c(y)
= a(x)y^2+ b(x)y+ c(x) .
\eeq
We shall also denote  
$$
K(x,0)Q(x,0)=R(x) \qquad \hbox{and} \qquad K(0,y) Q(0,y)=S(y) .
$$
Note that  $K(0,0) Q(0,0)=R(0)=S(0)$, so that the basic functional
equation~\eqref{eq:functional_equation} reads
\beq\label{eqfunc}
K(x,y)Q(x,y)=R(x)+S(y)-R(0)-xy.
\eeq

 Seen as  polynomial in $y$, the
kernel  has two roots $Y_0$ and $Y_1$, which are
Laurent series in $t$ with coefficients in $\qs(x)$. 
%Symmetrically, the kernel seen as a polynomial in $y$ admits two
%roots  $Y_0$ and  $Y_1$. 
%
If the series $Q(x,Y_i)$ is well defined, 
setting $y=Y_i$ in~\eqref{eqfunc} shows that
\begin{equation}
\label{eq:func_spec}
     R(x)+S(Y_i)= xY_i+R(0).
\end{equation}
If this holds for $Y_0$ and $Y_1$, %this implies
we have
\beq\label{SYi}
S(Y_0)-xY_0=S(Y_1)-xY_1.
\eeq
%Symmetrically, if both $Q(X_0,y)$ and $Q(X_1,y)$ are well defined,$$R(X_0)-yX_0=R(X_1)-yX_1.$$ 

The \emm group of the model,, denoted by  $G(\cS)$, is generated by the following two rational
transformations:
$$%\begin{equation}\label{eq:generators}
  \Phi(x,y) = \left(\frac{\widetilde c(y)}{\widetilde
      a(y)}\frac{1}{x},y\right) \qquad \hbox{and } \qquad    \Psi(x,y) = \left(x,\frac{c(x)}{a(x)}\frac{1}{y}\right)     .
$$%\end{equation}
Both are involutions, thus $G(\cS)$ is  a dihedral group, which,
depending on the step set $\cS$,  is finite or not.

A step set $\cS$ is \emm singular, if
each step $(i,j)\in \cS$ satisfies $i+j \ge 0$.

\medskip
\noindent {\bf Notation.} % We conclude this introduction with some notation.
For a ring $R$, we denote by $R[t]$
 (resp.~$R[[t]]$) 
the ring of
polynomials (resp.\ formal power series) 
in $t$ with coefficients in
$R$. If $R$ is a field, then $R(t)$ stands for the field of rational functions
in $t$. 
This notation is generalized to several variables.
For instance, the series $Q(x,y)$ %that counts quadrant walks is a
                                %series of
belongs to $\qs[x,y][[t]]$.

%%%%%%%%%%%%%%%%%%%%%%%%%%%%%%%%%%%%%%%%%%%%%%%%%%%%%%%%%%%%%%%
\section{A new solution of Gessel's model}
\label{sec:gessel}
%%%%%%%%%%%%%%%%%%%%%%%%%%%%%%%%%%%%%%%%%%%%%%%%%%%%%%%%%%%%%%%
This model, with steps $\rightarrow, \nearrow, \leftarrow, \swarrow$, appears as the most difficult  model
with a finite group. Around 2000, Ira Gessel conjectured that
the number of $2n$-step quadrant walks  starting and ending at $(0,0)$ was
$$%\beq\label{conj:gessel}
q(0,0;2n)=16^n\, \frac{ (1/2)_n(5/6)_n}{(2)_n(5/3)_n},
$$%\eeq
where $(a)_n=a(a+1) \cdots (a+n-1)$ is the %ascending
rising  factorial.
This 
conjecture was  proved in 2009 by Kauers, Koutschan and
Zeilberger~\cite{KaKoZe08}. A year later, Bostan and
Kauers~\cite{BoKa-10} proved that the 
three-variate series $Q(x,y;t)$ %of these walks
is not only D-finite, but even algebraic.
Two other proofs of these
results have  been
given~\cite{BKR-13,mbm-gessel}.
%, but none of them explains combinatorially the simplicity of the numbers, nor the algebraicity of the series.
Here, we give yet another proof based on Tutte's idea of \emm
% rational
 invariants,.

The basic functional equation~\eqref{eqfunc} holds with $K(x,y)=t (y +x^2y
 +x^2y^2+1)-xy$, 
$R(x)=tQ(x,0)$, 
and $S(y)=t(1+y)Q(0,y)$. It follows from  $K(x,Y_0)=K(x,Y_1)=0$ that
\beq\label{inv-I-G}
I(Y_0)=I(Y_1),  \qquad \hbox{ with } \qquad I(y)= \frac 1{t(1+y)(1+\by)}
+t(1+y)(1+\by).
\eeq
We say that $I(y)$ is a (rational) \emm $y$-invariant.,

Let us now take $x=t+t^2(u+1/u)$, where $u$ is a new variable. Then it is
easy to see that both $Y_0$ and $Y_1$ are  Laurent series in $t$
with coefficients in $\qs(u)$, 
and that  $Q(x,Y_0)$ and $Q(x,Y_1)$
are well defined. Hence~\eqref{SYi} holds. Moreover, the kernel 
equation $K(x,Y_i)=0$ implies that
\beq\label{decouple-gessel}
x(Y_0-Y_1)= \frac 1{t(1+Y_1)}- \frac 1{t(1+Y_0)},
\eeq
so that we can rewrite~\eqref{SYi} as
$$
J(Y_0)=J(Y_1), \qquad
\hbox{with } \qquad J(y)=S(y)+\frac{1}{t(1+y)}.
$$
This should be compared to~\eqref{inv-I-G}.
The connection between $I(y)$ and $J(y)$ will stem from the following
lemma, the proof of which is elementary.
\begin{lem}\label{lem:inv-gessel}
  Let $F(y)$
%\equiv F(y;t)$ 
be a Laurent series in $t$ with coefficients in $\qs[y]$, of the form 
$$
F(y)=\sum_{0\le j\le n+n_0} a(j,n) y^jt^{2n} 
$$
for some $n_0\ge 0$.
Then for 
%$u$ a formal variable and 
$x=t+t^2(u+1/u)$,  the series $F(Y_0)$
and $F(Y_1)$ are 
well defined Laurent series in $t$, with coefficients in
$\qs(u)$. 
If they coincide, then $F(y)$ is
in fact independent of $y$. 
% That is, $a(j,n)=0$ as soon as $j>0$. 
\end{lem}
The above series $I$ and $J$ do not satisfy the assumptions of the
lemma, as  their coefficients are \emm rational, in $y$ with
poles at $y=0, -1$ (for $I$) and $y=-1$ (for $J$). Still, we can
construct from them a series~$F$ 
satisfying the assumptions of the lemma. First, we eliminate the pole of $I$ at 0 by
forming the $y$-invariant $(J(y)-J(0))I(y)$.  The coefficients
of this series have a pole of order at most 3 at $y=-1$. By
subtracting an appropriate series of the form $
aJ(y)^3+bJ(y)^2+cJ(y)$, we obtain a series satisfying the assumptions
of the lemma, which must thus be constant, equal for instance at its
value at $y=-1$. In brief,
\beq\label{eq:IJ}
(J(y)-J(0))I(y)= aJ(y)^3+bJ(y)^2+cJ(y)+d
\eeq
for some series $a$, $b$, $c$, $d$ in $t$.
Expanding this identity near $y=-1$ gives:
$$%\begin{align*}
 a=-t, \qquad b=2+tS(0), \qquad c=-S(0)+ 2S'(-1)-1/t, 
$$%\end{align*}
and $$d=-2S(0)S'(-1)-3S'(-1)/t+S''(-1)/t.$$ 
Replacing in~\eqref{eq:IJ} the series $I$ and $J$ by their expressions (in terms of $t$, $y$ and
$S(y)$) gives  for $S(y)$ a cubic equation, involving three additional unknown series in $t$, namely $S(0)$,
$S'(-1)$ and $S''(-1)$. It is not hard to see that this equation
defines a unique power series $S(y)$ in $\qs[y][[t]]$. In the
terminology of~\cite{mbm-jehanne}, this is a \emm cubic equation  in one catalytic variable,
$y$. The solutions of such equations are always algebraic, and a
procedure for solving them is given in~\cite{mbm-jehanne}. Applying it
shows in particular 
that $S(0)/t$, $S'(-1)/t$ and $S''(-1)/t$ belong to  $\qs(Z)$, where
$Z$ is the  unique series in $t$ with constant term 1
%in $\qs[[t]]$, with constant term 1,
satisfying $Z^2=1+256t^2Z ^6/(Z^2+3)^3$. Due to lack
of space, we do not give any details, but refer the reader
to~\cite[Sec.~3.4]{mbm-gessel}, where an analogous
 equation satisfied by $R(x)$ is solved.

%%%%%%%%%%%%%%%%%%%%%%%%%%%%%%%%%%%%%%%%%%%%%%%%%%%%%%%%%%%%%%%%%%%%%%%%%%%%%%
\section{Extensions and obstructions}
\label{sec:extensions}
%%%%%%%%%%%%%%%%%%%%%%%%%%%%%%%%%%%%%%%%%%%%%%%%%%%%%%%%%%%%%%%%%%%%%%%%%%%%%%
We now formalize the three main  ingredients in the above solution of Gessel's
model. The first one is clearly the rational invariant $I(y)$ given
by~\eqref{inv-I-G}.
%We now consider an arbitrary model $\cS$ among the 79 ones with small steps.
\begin{defn}\label{def:rat-inv}
  A rational function $I(y) \in \qs(t,y)
\setminus \qs(t)$ 
is said to be a
  \emm $y$-invariant, of a quadrant model $\cS$  if  
$
I(Y_0) = I(Y_1)
$
when $Y_0$ and $Y_1$ are the roots of the kernel, solved for $y$. We
define $x$-invariants similarly. 
\end{defn}
Note that $I(Y_i)= (I(Y_0)+I(Y_1))/2$ must then be a
rational function of $x$, since it is a symmetric function of $Y_0$
and $Y_1$.
%, and that any rational function of an invariant is still an invariant.
\begin{lem}\label{lem:xy-inv}
  If a model has a
  $y$-invariant $I_2(y)$, then it admits $
I_1(x):= I_2(Y_0)= I_2(Y_1)$ as  $x$-invariant. Moreover, 
$ I_1(X_0)=I_1(X_1)=I_2(y).
$
\end{lem}
Note that having invariants is just saying that $I_1(x)-I_2(y)$ vanishes on
the curve $K(x,y)=0$, which alludes to Hilbert's Nullstellensatz. In Gessel's case, $I_2(y)$ was the function
$I(y)$ of~\eqref{inv-I-G}, and we find $I_1(x)=-t/x^2+1/x+2t+x-tx^2$.

\begin{prop}\label{prop:finite-invariant}
A quadrant model 
%with small steps 
has rational invariants if and only
 if the associated group is finite.  
\end{prop}
\begin{proof}
A model with an infinite group cannot have
rational invariants: the function $I_1(x)$ would take the same value
for infinitely many values of $x$, which is not possible for a
rational function. Conversely, each of the 23 models with a
finite group admits rational invariants, listed in Appendix~\ref{app:inv}. 
\end{proof}

In the next section, we introduce a weaker notion of (possibly
irrational) invariants, which guarantees that any 
 quadrant
model now has a (weak) invariant. One key difference with the
algebraic setting of the above section is that the new notion is
analytic in nature.

  \medskip
The second ingredient of Section~\ref{sec:gessel} was the
identity~\eqref{decouple-gessel}, which we formalize as follows.
\begin{defn}\label{def:decoupled}
  A quadrant model  is \emm decoupled, if there exist
  rational functions $F(x) \in \qs(x,t)$ and $G(y)\in \qs(y,t)$ such
  that, as soon as $K(x,y)=0$, one has
$
xy=F(x)+G(y).
$
\end{defn}
This is in fact equivalent to the (apparently weaker) identity
$x(Y_0-Y_1)= G(Y_0)-G(Y_1)$ (which was~\eqref{decouple-gessel} in
Gessel's case), and  
a possible choice for $F$ is then 
 $F(x)=xY_0-G(Y_0)=xY_1-G(Y_1)$.
% (and symmetrically). 
In Gessel's case, we had $G(y)=-1/(t(1+y))$,
corresponding to $F(x)=1/t-1/x$.

So which models are decoupled? Not all, at any rate: for any model
that has a vertical symmetry, the series $Y_i$ are symmetric in $x$
and $\bx$, and so any expression of $x$ of the form $
(G(Y_0)-G(Y_1))/(Y_0-Y_1)$ would be at the same time an expression of
$1/x$. At the moment, we have found 13 decoupled models,
shown in Tables~\ref{tab:decoupling_functions-finite} and~\ref{tab:decoupling_functions-infinite}: 4 with a 
finite group (and these are, as one can expect from the algebraicity
result of Section~\ref{sec:gessel}, those with an
algebraic \gf), and 9 with an infinite group.
%, shown in Figure~\ref{fig:infinite_models}. 

\begin{table}[ht]
\begin{center}
 \small
\begin{tabular}{|c|c|c||c|c|c|c|}
\hline
Model & \begin{tikzpicture}[scale=.2] % q6654
    \draw[->] (0,0) -- (1,1);
    \draw[->] (0,0) -- (-1,0);
    \draw[->] (0,0) -- (0,-1);
        \draw[-] (0,-1) -- (0,-1) node[below] {\phantom{$\scriptstyle 1$}};
  \end{tikzpicture} \  \begin{tikzpicture}[scale=.2] % q6654
    \draw[->] (0,0) -- (0,1);
    \draw[->] (0,0) -- (-1,-1);
    \draw[->] (0,0) -- (1,0);
    \draw[-] (0,-1) -- (0,-1) node[below] {\phantom{$\scriptstyle 1$}};
  \end{tikzpicture} \     \begin{tikzpicture}[scale=.2] % q6654
      \draw[->] (0,0) -- (0,-1);
    \draw[->] (0,0) -- (1,1);
    \draw[->] (0,0) -- (-1,0);
    \draw[->] (0,0) -- (0,1);
    \draw[->] (0,0) -- (-1,-1);
    \draw[->] (0,0) -- (1,0);
    \draw[-] (0,-1) -- (0,-1) node[below] {\phantom{$\scriptstyle 1$}};
  \end{tikzpicture} &   \begin{tikzpicture}[scale=.2] % q6654
      \draw[->] (0,0) -- (-1,-1);
    \draw[->] (0,0) -- (1,1);
    \draw[->] (0,0) -- (-1,0);
    \draw[->] (0,0) -- (1,1);
    \draw[->] (0,0) -- (1,0);
    \draw[-] (0,-1) -- (0,-1) node[below] {\phantom{$\scriptstyle 1$}};
  \end{tikzpicture} & \begin{tikzpicture}[scale=.2] % q6654
    \draw[->] (0,0) -- (-1,0) node[left] {$\scriptstyle 1$};
    \draw[->] (0,0) -- (-1,-1) node[left] {$\scriptstyle 1$};
    \draw[->] (0,0) -- (0,-1) node[below] {$\scriptstyle \lambda$};
    \draw[->] (0,0) -- (1,-1) node[right] {$\scriptstyle 1$};
    \draw[->] (0,0) -- (1,0) node[right] {$\scriptstyle 2$};
    \draw[->] (0,0) -- (1,1) node[right] {$\scriptstyle 1$};
  \end{tikzpicture} & \begin{tikzpicture}[scale=.2] % q6654
    \draw[->] (0,0) -- (-1,0) node[left] {$\scriptstyle 1$};
    \draw[->] (0,0) -- (-1,1) node[left] {$\scriptstyle 1$};
    \draw[->] (0,0) -- (0,1) node[above] {$\scriptstyle 2$};
    \draw[->] (0,0) -- (1,1) node[right] {$\scriptstyle 1$};
    \draw[->] (0,0) -- (1,0) node[right] {$\scriptstyle 2$};
    \draw[->] (0,0) -- (1,-1) node[right] {$\scriptstyle 1$};
    \draw[->] (0,0) -- (0,-1) node[below] {$\scriptstyle 1$};
  \end{tikzpicture}&
  \begin{tikzpicture}[scale=.2] % q6654
    \draw[->] (0,0) -- (-1,0) node[left] {$\scriptstyle 2$};
    \draw[->] (0,0) -- (-1,1) node[left] {$\scriptstyle 1$};
    \draw[->] (0,0) -- (0,1) node[above] {$\scriptstyle 1$};
    \draw[->] (0,0) -- (-1,-1) node[left] {$\scriptstyle 1$};
    \draw[->] (0,0) -- (1,0) node[right] {$\scriptstyle 1$};
    \draw[->] (0,0) -- (1,-1) node[right] {$\scriptstyle 1$};
    \draw[->] (0,0) -- (0,-1) node[below] {$\scriptstyle 2$};
  \end{tikzpicture} & \begin{tikzpicture}[scale=.2] % q6654
    \draw[->] (0,0) -- (-1,0) node[left] {$\scriptstyle 2$};
    \draw[->] (0,0) -- (-1,1) node[left] {$\scriptstyle 1$};
    \draw[->] (0,0) -- (0,1) node[above] {$\scriptstyle 2$};
    \draw[->] (0,0) -- (1,1) node[right] {$\scriptstyle 1$};
    \draw[->] (0,0) -- (1,0) node[right] {$\scriptstyle 1$};
    \draw[->] (0,0) -- (0,-1) node[below] {$\scriptstyle 1$};
        \draw[->] (0,0) -- (-1,-1) node[left] {$\scriptstyle 1$};
  \end{tikzpicture}  \\
\hline
%
%$F$& $-\frac{1}{x} $ &$-\frac{1}{x}  $ & $-x-\frac{1}{x}$ & $-x+\frac{1}{x}-\frac{1+3t}{t(1+x)}$ &  $-x^2 +x(1+\frac{1}{t})+\frac{3+\frac{1}{t}}{x}$&$-x-\frac{1+3t}{t(1+x)}$\\
%\hline
$G$& $-\frac{1}{y}$
 &$- \frac1{t(1+y)}$ &$-\frac{1+\lambda t}{t(1+y)}$
& $-y+\frac{1}{y}-\frac{1+3t}{t(1+y)}$ &  $-y^2
+y(1+\frac{1}{t})+\frac{3+\frac{1}{t}}{y}$ & $-y-\frac{1}{y}$ \\ 
\hline
\end{tabular}
\end{center}
\normalsize
  \caption{Decoupling functions for algebraic models
(unweighted or weighted). Recall that $F(x)=xY_i-G(Y_i)$.}
  \label{tab:decoupling_functions-finite}
\end{table}

%\vspace{-50mm}

 \begin{table}[h!]
 \small
\begin{center}
\begin{tabular}{|c|c|c|c|c|c|c|c|}
\hline
&&&&&&&\\
Model & \begin{tikzpicture}[scale=.2] % q6654
    \draw[->] (0,0) -- (0,1);
    \draw[->] (0,0) -- (1,0);
    \draw[->] (0,0) -- (-1,0);
    \draw[->] (0,0) -- (-1,-1);
    \draw[-] (0,-1) -- (0,-1) node[below] {$\scriptstyle \# 1$};
  \end{tikzpicture} &   \begin{tikzpicture}[scale=.2] % q6654
    \draw[->] (0,0) -- (0,1);
    \draw[->] (0,0) -- (1,0);
    \draw[->] (0,0) -- (-1,1);
    \draw[->] (0,0) -- (-1,-1);
    \draw[-] (0,-1) -- (0,-1) node[below] {$\scriptstyle \# 2$};
  \end{tikzpicture} &     \begin{tikzpicture}[scale=.2] % q6654
    \draw[->] (0,0) -- (0,1);
    \draw[->] (0,0) -- (1,1);
    \draw[->] (0,0) -- (0,-1);
    \draw[->] (0,0) -- (-1,0);
    \draw[-] (0,-1) -- (0,-1) node[below] {$\scriptstyle \# 3$};
  \end{tikzpicture} & \begin{tikzpicture}[scale=.2] % q6654
    \draw[->] (0,0) -- (0,1);
    \draw[->] (0,0) -- (1,0);
    \draw[->] (0,0) -- (1,-1);
    \draw[->] (0,0) -- (-1,0);
    \draw[-] (0,-1) -- (0,-1) node[below] {$\scriptstyle \# 4$};
  \end{tikzpicture} &     \begin{tikzpicture}[scale=.2] % q6654
    \draw[->] (0,0) -- (0,1);
    \draw[->] (0,0) -- (1,0);
    \draw[->] (0,0) -- (1,1);
    \draw[->] (0,0) -- (-1,-1);
    \draw[->] (0,0) -- (-1,0);
    \draw[-] (0,-1) -- (0,-1) node[below] {$\scriptstyle \# 5$};
  \end{tikzpicture} &       \begin{tikzpicture}[scale=.2] % q6654
    \draw[->] (0,0) -- (0,1);
    \draw[->] (0,0) -- (0,-1);
    \draw[->] (0,0) -- (1,1);
    \draw[->] (0,0) -- (-1,-1);
    \draw[->] (0,0) -- (-1,0);
    \draw[-] (0,-1) -- (0,-1) node[below] {$\scriptstyle \# 6$};
  \end{tikzpicture} &     \begin{tikzpicture}[scale=.2] % q6654
    \draw[->] (0,0) -- (-1,1);
    \draw[->] (0,0) -- (-1,0);
    \draw[->] (0,0) -- (1,0);
    \draw[->] (0,0) -- (-1,-1);
    \draw[->] (0,0) -- (0,1);
    \draw[-] (0,-1) -- (0,-1) node[below] {$\scriptstyle \# 7$};
  \end{tikzpicture}  \\
\hline
%$F$& $-x^2+\frac{x}{t}$ &$-x^2+\frac{x}{t}$ &$-\frac{1}{x}$ &$\frac{1}{x^2}-\frac{1}{xt}-x$ & $-\frac{1}{x}$ &$-\frac{1+t}{t(x+1)}$ & $-x^2+\frac{x}{t}$\\\hline
$G$& $-\frac{1}{y}$ &$-y-\frac{1}{y}$ &$-y-\frac{1}{y}$ & $-y^2+\frac{y}{t}+\frac{1}{y}$&$-\frac{1+t}{t(y+1)}-y$ & $-\frac{1}{y}$&$-y-\frac{1}{y}$   \\
\hline
\end{tabular}

\medskip
\begin{tabular}{|c|c|}
\hline
&\\
 \begin{tikzpicture}[scale=.2] % q6654
    \draw[->] (0,0) -- (1,1);
    \draw[->] (0,0) -- (0,-1);
    \draw[->] (0,0) -- (1,0);
    \draw[->] (0,0) -- (-1,0);
    \draw[->] (0,0) -- (0,1);
    \draw[-] (0,-1) -- (0,-1) node[below] {$\scriptstyle \# 8$};
  \end{tikzpicture} &    \begin{tikzpicture}[scale=.2] % q6654
    \draw[->] (0,0) -- (1,0);
    \draw[->] (0,0) -- (0,-1);
    \draw[->] (0,0) -- (0,1);
    \draw[->] (0,0) -- (-1,1);
    \draw[->] (0,0) -- (1,-1);
    \draw[-] (0,-1) -- (0,-1) node[below] {$\scriptstyle \# 9$};
  \end{tikzpicture} \\
\hline
%$-\dfrac{1}{x}-x$ &$\dfrac{(t+1)^2}{(x+1)^2t^2}-\dfrac{2t^2+3t+1}{(x+1)t^2}-\dfrac{1}{x}-x$  \\\hline
$-\dfrac{1}{y}-y$  &$-\dfrac{1}{y^2}+\dfrac{1}{ty}+\dfrac{(t+1)y}{t}-y^2$ \\
\hline
\end{tabular}
\normalsize
\end{center}
  \caption{Decoupling functions for nine infinite group models.}
  \label{tab:decoupling_functions-infinite}
\end{table}

The final ingredient is the ``invariant Lemma''
(Lemma~\ref{lem:inv-gessel}). It admits  analogues for the 4 algebraic
models to the left of 
%Figure~\ref{fig:finite_models}, 
Table~\ref{tab:decoupling_functions-finite},
except for the second one. 
For this model, we were unable to find an analogue of the invariant lemma, and we solve it in a slightly different way (using a weaker form of the invariant lemma) in the next section.
\begin{prop}
  The existence of rational invariants,  decoupling  functions and
invariant lemma yields new and uniform proofs of the algebraicity of
the $4$ leftmost models  of 
%Figure~\ref{fig:finite_models}, left,
Table~\ref{tab:decoupling_functions-finite},
  except the second
(called reverse Kreweras model).
 This extends to the $4$ weighted models 
shown on the right of this table.
%  Figure~\ref{fig:finite_models}, right.

\end{prop}
For Gessel's walks, this is the shortest known proof. For the
rightmost weighted model, it is the first one. 

%The invariant Lemma will be adapted to weak invariants in the next section (Lemma~\ref{lem:inv-analytic}).

\medskip
\noindent{\bf Remark.} 
 In the finite group case, there exists a systematic procedure to construct
the invariants and (when they exist) the
decoupling functions,
adapting~\cite[Thm.~4.2.9 and Thm.~4.2.10]{FIM-99} to our context. 
In the infinite group case, the 9 decoupling functions have been guessed-and-checked.
\section{An analytic invariant method}
\label{sec:analysis}
%%%%%%%%%%%%%%%%%%%%%%%%%%%%%%%%%%%%%%%%%%%%%%%%%%%%%%%%%%%%%%%%%%%%%%%%%%%%%%
We now move to an analytic world, and consider $Q(x,y)$ as a
function of three complex variables,  analytic in the
polydisc $\{|x|< 1, |y|<1, |t|<1/|\cS|\}$ (at least). 
This section borrows its notation and several results from the analytic
approach of quadrant models~\cite{FIM-99,Ra-12}.
 The roots
$Y_{0,1}$ of the kernel (now called \emm branches, of $K$) are %given by
$$%\begin{equation}\label{eq:formula_Y_12}
   %  X_{0,1}(y)=\frac{-\widetilde b(y)\pm\sqrt{\widetilde
   %  b(y)^2-4\widetilde a(y)\widetilde c(y)}}{2\widetilde a(y)},
    Y_{0,1}(x)=\frac{- b(x)\pm\sqrt{ b(x)^2-4 a(x) c(x)}}{2 a(x)},
$$%\end{equation}
where $a, b$ and $c$ are defined by~\eqref{K-abc}.  The discriminant  $ d(x):=
 b(x)^2-4 a(x) c(x)$ has degree three or four,
hence there are four branch points 
% $x_\ell$ 
$x_1, \ldots, x_4$ 
(depending on $t$), with
$x_4=\infty$ if $ d(x)$ has degree three.   We define symmetrically
the branches $X_{0,1}(y)$ and  their four branch 
points $y_\ell$.
\begin{lem}[\cite{FIM-99}]
\label{lem:branch_points}
Let $t\in (0,1/\vert \mathcal S\vert)$. The $x_\ell$'s are real. Two
branch points (say $x_1$ and $x_2$) are in the (open) unit disc, with
$\vert x_1\vert \leq \vert x_2\vert$ and $x_2>0$. The other two  (say
$x_3$ and $x_4$) are outside the (closed) unit disc, with $\vert
x_3\vert \leq \vert x_4\vert$ and $x_3>0$. The discriminant
$ d(x)$ is negative on $(x_1,x_2)$ and $(x_3,x_4)$, where if $x_4<0$, 
the interval 
$(x_3,x_4)$ stands for $(x_3,\infty)\cup (-\infty,x_4)$.
\end{lem} 

The branches $Y_{0,1}$ are meromorphic on $\mathbb C\setminus
([x_1,x_2]\cup [x_3,x_4])$. 
On the cuts $[x_1,x_2]$ and  $[x_3,x_4]$,
the two branches $Y_{0,1}$ 
still exist and are complex conjugate. A key object
in our definition of weak invariants 
is the  curve  $\mathcal L$ 
(depending on $t$)
defined by
$$%\begin{equation}\label{eq:curve_X}
     \mathcal L =Y_0([x_1,x_2])\cup Y_1([x_1,x_2])=\{y\in \mathbb C:
     K(x,y)=0 \text{ and } x\in[x_1,x_2]\}. 
$$%\end{equation}
By construction, it is symmetric with respect to the real axis.  If
$\mathcal L$ is  bounded (as for the models of
%Figure~\ref{fig:infinite_models}), 
Table~\ref{tab:decoupling_functions-infinite}),
 we denote by $\mathcal G_\mathcal
L$  the  domain enclosed by $\mathcal L$. 
Otherwise, $\mathcal G_\mathcal L$ is the domain delimited by $\mathcal L$ and containing
the real point at $-\infty$.
 See Figure~\ref{fig:curves} for examples.

\begin{figure}[ht]
  \includegraphics[width=6cm,height=6cm]{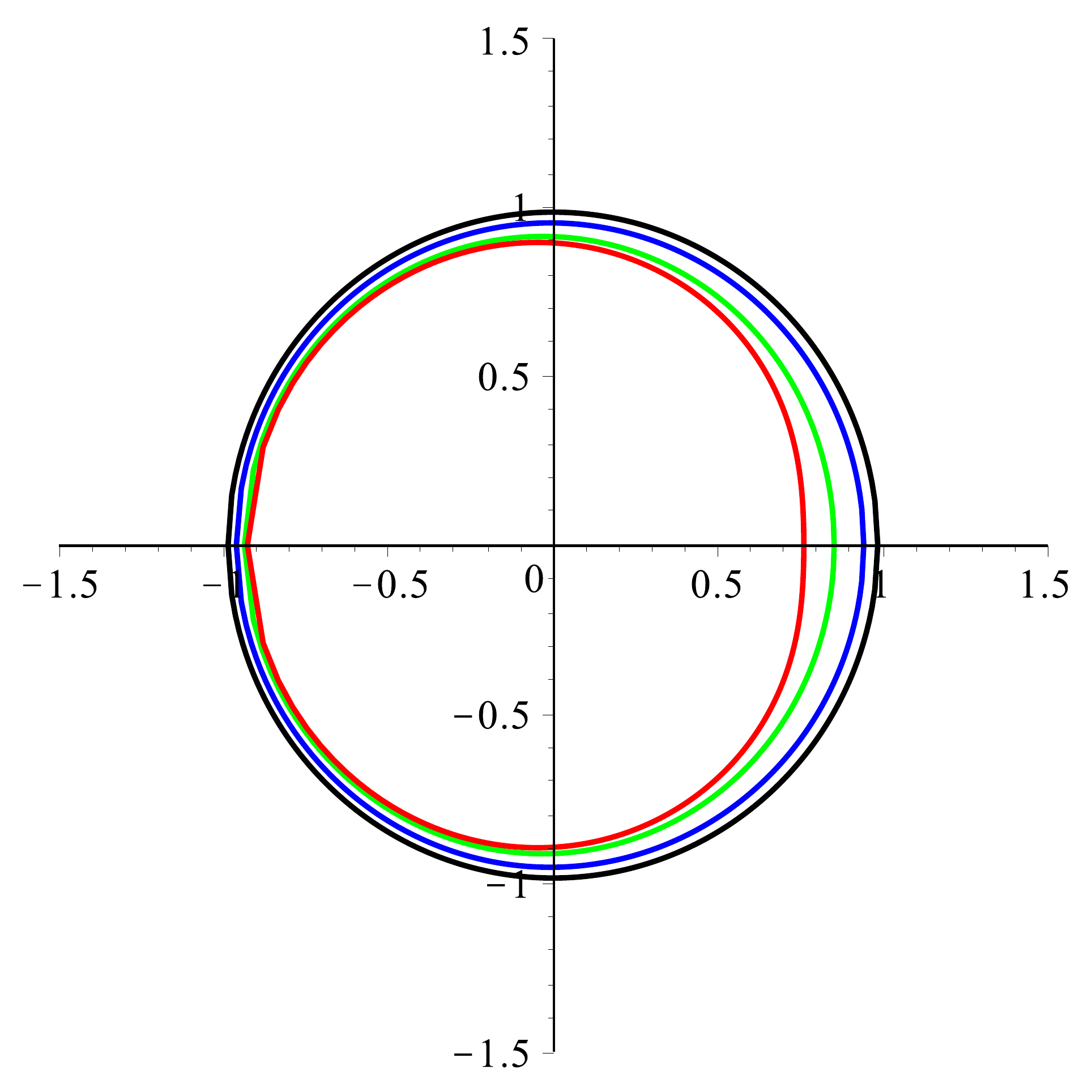} 
 \hspace{24mm}
\includegraphics[width=6cm,height=6cm]{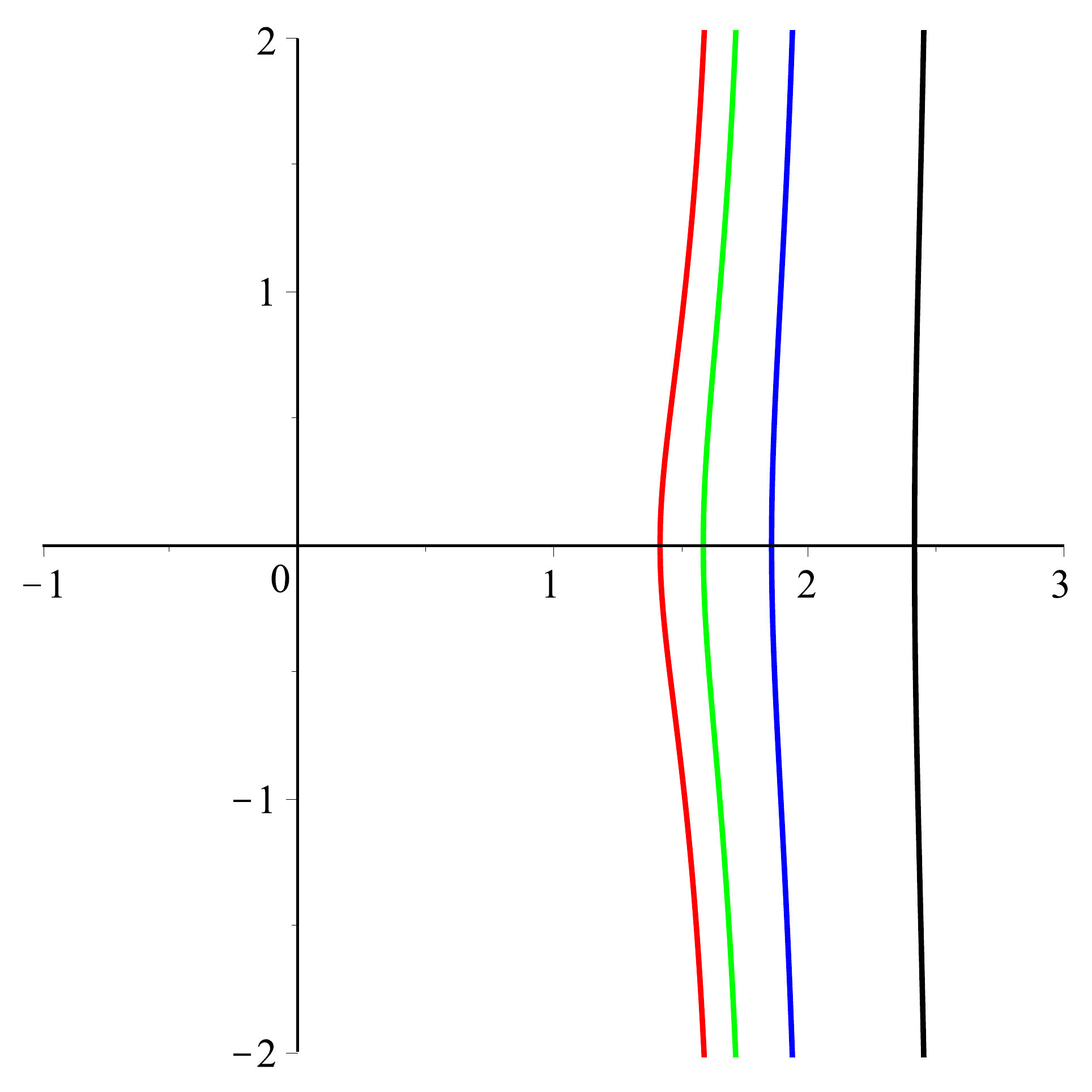}
\caption{The curves $\mathcal L$ for model ${\#3}$
of Table~\ref{tab:decoupling_functions-infinite}
  (for $t=0.03$,
% (black), 
$0.1$,
% (blue), 
$0.2$ % (green) 
and $0.25$, 
% (red)
as one moves closer to the origin) and for the reverse Kreweras model
(second model in Table~\ref{tab:decoupling_functions-finite};
$t=0.2$,
% (black), 
$0.25$,
% (blue), 
$0.28$
% (green) 
and $0.33$,
% (red)
from  right to left).}
\label{fig:curves}
\end{figure}

The function $S(y)=K(0,y)Q(0,y)$ is analytic
in $\mathcal G_\mathcal L$ (by~\cite[Thm.~5]{Ra-12})
and bounded on $\mathcal G_\mathcal L\cup\mathcal L$ 
(this follows from \eqref{eq:func_spec} when $x\in[x_1,x_2]$). 
Moreover, Eq.~(11) of~\cite{Ra-12} tells us that, for $x \in [x_1, x_2]$,
 \beq\label{eqker:Y01}
%R(X_0)-yX_0= R(X_1)-yX_1.
S(Y_0)-xY_0=S(Y_1)-xY_1.
\eeq

%A symmetric statement holds of course for $S(y)=K(0,y)Q(0,y)$.

%==============================================
\subsection{Weak invariants}
%==============================================

\begin{defn}
\label{def:weak-inv} 
 A function
$I(y)\equiv I(y;t)$ is %called
 a \emm  weak invariant, of a quadrant model
$\cS$ if for $t \in (0,1/|\cS|)$:
\begin{itemize}
\item it is meromorphic in the domain $\mathcal G_\mathcal L$,
and  admits finite limit values on the curve $\mathcal L$;
\item for any $y\in\mathcal L$, we have $I(y)=I(\overline{y})$.
\end{itemize}
\end{defn}

The second condition is indeed a weak form of the invariant condition
$I(Y_0)=I(Y_1)$, because two conjugate points $y$ and $\overline{y}$ of the
curve $\mathcal L$ are the (complex conjugate) roots of $K(x,y)=0$ for some
$x\in[x_1,x_2]$. Hence, if the model admits a rational invariant $I(y)$ in the
sense of Definition~\ref{def:rat-inv}, then $I$ is also a weak
invariant. However, the above definition is less demanding, and it
turns out that every quadrant model admits a 
(non-trivial)
weak
invariant, which we now describe 
(in the non-singular case).

This invariant, traditionally denoted $w(y)$ 
(or even $\widetilde w (y)$) 
in the analytic approach
to quadrant problems~\cite{FIM-99,Ra-12}, is in addition injective in  $\mathcal
G_\mathcal L$. In analytic terms, this third condition makes it a \emm
conformal gluing function, for the domain $\mathcal G_\mathcal
L$. Explicit expressions of conformal gluing functions are known in a
number of cases (when the domain is an ellipse, a polygon, etc.).  In
our case the bounding curve $\mathcal L$ is a quartic
curve~\cite[Thm.~5.3.3 (i)]{FIM-99}, and  $w$ can be expressed in
terms of Weierstrass' elliptic
functions~(see~\cite[Sec.~5.5.2.1]{FIM-99} or~\cite[Thm.~6]{Ra-12}): 
\begin{equation}
\label{eq:expression_gluing}
     w(y;t)\equiv
     w(y)=\wp_{1,3}\Big(-\frac{\omega_1+\omega_2}{2}+\wp_{1,2}^{-1}(f(y))\Big), 
\end{equation}
where the various ingredients 
of this expression
are as follows. First, $f(y)$ is a
simple rational function of $y$ whose
coefficients are algebraic functions of  $t$:
$$%\begin{equation}\label{eq:def_f}
     f(y) = \left\{\begin{array}{ll}
     \displaystyle \frac{\widetilde d''(y_4)}{6}+\frac{\widetilde d'(y_4)}{y-y_4} & \text{if } y_4\neq \infty,\medskip\\
     \displaystyle\frac{\widetilde d''(0)}{6}+\frac{\widetilde d'''(0)y}{6}& \text{if } y_4=\infty,
     \end{array}\right.
$$%\end{equation}
where the $y_\ell$'s are the branch points of the functions $X_{0,1}$,
and $\widetilde d(y)$ is the counterpart of the discriminant $ d(x)$ for the
variable  $y$ (so that $\widetilde d(y_4)=0$). 
Then, $\wp_{1,2}$ (resp.~$\wp_{1,3}$) is Weierstrass' elliptic function
with  periods $(\omega_1,\omega_2)$ (resp.~$(\omega_1,\omega_3)$),
where
\begin{equation}
\label{eq:expression_periods}
     \omega_1 = i\int_{x_1}^{x_2} \frac{\text{d} x}{\sqrt{- d(x)}},\quad
     \omega_2 = \int_{x_2}^{x_3} \frac{\text{d} x}{\sqrt{ d(x)}},\quad
     \omega_3 = \int_{Y(x_1)}^{x_1} \frac{\text{d} x}{\sqrt{ d(x)}}.
\end{equation}
%Here, the $x_\ell$'s are the branch points of the functions
%$X_{0,1}$, defined at the beginning of the section. 
Note that $\omega_1\in i\mathbb R_+$
and $\omega_2,\omega_3\in \mathbb R_+$. 

It is known that the function $w(y)$ given by~\eqref{eq:expression_gluing} 
is meromorphic in $\mathcal G_\mathcal L$, with a unique pole at
%$y_2\in\mathcal G_\mathcal L$.
$y_2$ (which lies indeed in $\mathcal G_\mathcal L$).
It is an algebraic function of $y$ and $t$ for the
23 models with a finite group, see~\cite[Thm.~2 and Thm.~3]{Ra-12}. It
is even rational unless $\cS$ is one of 
the 4 algebraic models. %In this case
 It is then a rational invariant, in
the sense of Definition~\ref{def:rat-inv}. 
In the infinite group case, $w(y)$ is not algebraic, nor even
D-finite w.r.t.\ to $y$, see~\cite[Thm.~2]{Ra-12}. However, we will
prove in Proposition~\ref{prop:w-DA} 
that it is differentially algebraic  in $y$ and $t$.

%==============================================
\subsection{The analytic  invariant lemma --- Application to quadrant walks}
%==============================================
We now come with an  analytic counterpart of Lemma~\ref{lem:inv-gessel}.
\begin{lem}\label{lem:inv-analytic}
  Let $\cS$ be a non-singular quadrant model and $I(y)$ a weak
  invariant for this model. If $I$ has no pole in $\mathcal G_\mathcal
  L$ (and, in the case of a non-bounded curve $\mathcal L$, if $I$ is bounded at $\infty$), it is independent of $y$. 
\end{lem}
 This is proved in~\cite[Ch.~3]{Li-00},
in Lemma~1 (resp.\ 2) for
 the bounded (resp.\ unbounded) case.
%, and in Lemma~2 for the unbounded case. 

%==============================================
Let $\cS$ be a quadrant model that is decoupled, in the sense of
Definition~\ref{def:decoupled}. In particular,
$$
%y(X_0-X_1)= F(X_0)-F(X_1)
x(Y_0-Y_1)=G(Y_0)-G(Y_1)
$$
for some rational function $G$. 
Assume that $G$ has no pole on $\mathcal L$.
Then, by combining~\eqref{eqker:Y01} with the
analytic properties of $S$, we see that 
$$
I(y):=S(y)-G(y) % (x)-F(x)
$$
is a weak invariant in the sense of Definition~\ref{def:weak-inv}. Since $S$ is
analytic in $\mathcal G_\mathcal L$, the 
poles of $I(y)$ lying in $\mathcal G_\mathcal L$ %
must be poles of $G(y)$.
 Let us denote them $p_1, \ldots, p_\ell$,
and assume they are different from the pole $y_2$ of~$w$.
 Then there exists a
function of the form 
$$
r(y)= \sum_{i=1}^\ell  \sum_{e=1}^{m_i}
\frac{\alpha_{e,i}}{(w(y)-w(p_i))^e}
$$
such that 
$I(y)-r(y)$ has no pole in $\mathcal G_\mathcal L$ and is bounded there 
--- and is still a weak invariant. Applying  Lemma~\ref{lem:inv-analytic} tells us that this function is
independent of $y$. Let us examine two  examples.

\medskip
\noindent{\bf Example: Model $\boldsymbol {\#3}$ of 
%Figure  \ref{fig:infinite_models}
Table~\ref{tab:decoupling_functions-infinite}.} 
This is a decoupled model, with
$G(y)=-y-1/y$. Hence $I(y)=S(y)+y+1/y= tyQ(0,y)+y+1/y$ is a weak invariant,
with a single pole in $\mathcal G_\mathcal L$, placed at $y=0$ and
having residue $1$ (the
curve $\mathcal G_\mathcal L$ is shown in Figure~\ref{fig:curves}, left). Thus
$I(y)$ differs from $\frac{w'(0)}{w(y)-w(0)}$ by a constant, and a
series expansion at $y=0$ gives
 \begin{equation}\label{eq:Q(x,0)_Kreweras_model_6}
     t\, Q(0,y) =
     \frac{1}{y}\left(\frac{w''(0)}{2w'(0)}+\frac{w'(0)}{w(y)-w(0)}-y-\frac{1}{y}\right). \end{equation} 
\qed 

This argument
extends to all models of %Figure~\ref{fig:infinite_models}.
Table~\ref{tab:decoupling_functions-infinite}.
The curve $\mathcal L$ is  bounded in each case.

\begin{cor}\label{cor:9models}
  For any of the $9$ models of 
%Figure~\ref{fig:infinite_models}, 
Table~\ref{tab:decoupling_functions-infinite}, 
the   series $tQ(0,y)$ 
  admits a rational expression in terms of $y$, $w(y)$, $w(0)$,
  $w'(0)$, $w''(0)$, $w(-1)$, $w'(-1)$, $w''(-1)$, $w(\pm i)$, $w(\pm
  j)$, with $j= e^{2i \pi/3}$. 
\end{cor}

%\medskip
%\noindent{\bf Example: Model $\boldsymbol {\#7}$ of Figure
%  \ref{fig:infinite_models}.} 
%This is a decoupled model, with
%$F(x)=-x-1/x$. Hence $I(x)=R(x)+x+1/x= t(1+x+x^2)Q(x,0)+x+1/x$ is a weak invariant,
%with a single pole in  $\mathcal G_\mathcal L$, placed at $x=0$ and
%having residue $1$ (the
%curve  $\mathcal G_\mathcal L$ is shown in Figure~\ref{fig:curves}, right). Thus
%$I(x)$ differs from $\frac{w'(0)}{w(x)-w(0)}$ by a constant, and a
%series expansion at $x=j$, where $j$ is any of the primitive cubic
%roots of unity,  gives
%\begin{equation*}
%     Q(x,0)=\frac{1}{t(1+x+x^2)}\left(\frac{w'(0)}{w(x)-w(0)}-\frac{1}{x}-1-x-\frac{w'(0)}{w( j)-w(0)}\right).
%\end{equation*}

\medskip
\noindent{\bf Example: the reversed Kreweras model.}
Recall that this model was left unsolved in Section~\ref{sec:extensions}. The
associated curve $\mathcal L$ is unbounded
(Figure~\ref{fig:curves}, right). 
This is again a decoupled model, 
and a 
decoupling function is $G(y)=-1/y$ (Table~\ref{tab:decoupling_functions-finite}).
Accordingly, $I(y)=S(y)+1/y= tQ(0,y)+1/y$ is a weak invariant,
with a single pole in $\mathcal G_\mathcal L$, 
placed at $y=0$ and having residue $1$. 
%Further, it follows from \eqref{eq:func_spec}
%that $I$ is bounded at $\infty$. 
Further, one can derive from~\eqref{eq:func_spec}
that $I$ is bounded at $\infty$. 

Thus $I(y)$ differs from
$\frac{w'(0)}{w(y)-w(0)}$ by a constant, and a series expansion at
$y=0$ gives
$$
tQ(0,y)=-\frac 1  y + \frac{w'(0)}{w(y)-w(0)} + tQ(0,0) + \frac { w''(0)}{2w'(0)}.
$$
It remains to determine $Q(0,0)$. Using the case $i=0$ of~\eqref{eq:func_spec}
and the $x/y$ symmetry, we find that
$$
xY_0%= S(x)+S(Y_0)-xY_0 -tQ(0,0)
= -\frac 1  x + \frac{w'(0)}{w(x)-w(0)}-\frac 1  {Y_0}+\frac{w'(0)}{w(Y_0)-w(0)} 
 + tQ(0,0) + \frac { w''(0)}{w'(0)}.
$$
%   Instead, we use themain functional equation
%   \eqref{eq:functional_equation}: $tQ(0,0)$
%   equals$$xY_0(x)-(R(x)-R(0))-(S(Y_0(x))-S(0))=-x^2+\frac{x}{t}+\frac{1}{x}-\frac{w'(0)}{w(y)-w(0)}-\frac{w'(0)}{w(Y_0(x))-w(0)}.$$ The right-hand side of the above equality is independent of $x$, and can be evaluated (e.g., at $x=1$) to provide a simpler expression for $Q(0,0)$. 
Specializing for instance at $x=1$ gives an expression of $tQ(0,0)$.
In order to prove the algebraicity of $Q(0,y)$ (and hence that of $Q(x,y)$
by~\eqref{eq:functional_equation}), it now suffices to prove that $w(y)$
is algebraic. This can be done by applying
Lemma~\ref{lem:inv-analytic} to the rational invariant 
$I_2(y)={t}{y^2}-{y}-t/y$. This gives for $w(y)$ an equation with one
catalytic variable, which can be solved in a systematic way using~\cite{mbm-jehanne}. 
\qed

%%%%%%%%%%%%%%%%%%%%%%%%%%%%%%%%%%%%%%%%%%%%%%%%%%%%%%%%%%%%%%%%%%%%%%%%%%%%%%
\section{Differential algebraicity}
\label{sec:DA}
%%%%%%%%%%%%%%%%%%%%%%%%%%%%%%%%%%%%%%%%%%%%%%%%%%%%%%%%%%%%%%%%%%%%%%%%%%%%%%
As recalled in the introduction,  quadrant walks have a D-finite \gf\ if and only if the
associated group is finite
%This has been  proved in two different ways for the  51 non-singular
%models~\cite{KR-12,BoRaSa12}, and in a more ad hoc fashion for the
%remaining 5 ones~\cite{MiRe09,MeMi13}.  
--- we can now say, if and only if they admit a rational invariant
(Proposition~\ref{prop:finite-invariant}).  
Still, one outcome of our analytic invariant approach is
that \emm non-linear, differential equations hold for a number of models
with an infinite group.

\begin{thm}
\label{thm:DA}
For any of the $9$  models %with infinite group 
of 
%Figure~{\rm\ref{fig:infinite_models}}, 
Table~\ref{tab:decoupling_functions-infinite}, 
the generating function $Q(x,y;t)$ is
\emm differentially algebraic, (or \emm D-algebraic,) in $x,y,t$. By this, we mean that it
satisfies \emm three, (non-linear) differential equations with coefficients in $\qs$: one in $x$,
one in $y$ and one in $t$. 
\end{thm}

The proof builds on Corollary~\ref{cor:9models} and on the following
result, which holds for any non-singular model.

\begin{prop}\label{prop:w-DA}
  The conformal gluing function $w(y;t)$ defined by~{\rm\eqref{eq:expression_gluing}} is
   D-algebraic in $y$ and~$t$.
\end{prop}
{\bf Proof:} (sketched). There are three main steps in the proof:
  \begin{itemize}
  \item We first consider the Weierstrass function
    $\wp(\omega)$  as a function of $\omega$, but also as a
    function of its periods $\om_1$ and $\om_2$ or, alternatively, of
    the values  $g_2$ and $g_3$ (also called \emm invariants, in the
    elliptic terminology!) defined by
$$
g_{k}(\om_1,\om_2)= \sum_{(i,j) \in \zs^2\setminus\{ (0,0)\}}  \frac 1{(i \omega_1
  +j \omega_2)^{2k}}.
$$
Using some  differentiation formul\ae\ borrowed from~\cite{AS},  we prove that $\wp$ is DA in $\om$, $g_2$ and~$g_3$.
\item Then, we prove that when $\om_1$, $\om_2$ are either the functions
  $\om_1(t)$, $\om_2(t)$ given by~\eqref{eq:expression_periods}, or the functions
  $\om_1(t)$, $\om_3(t)$ (still given by~\eqref{eq:expression_periods}), then 
  the functions $g_2$ and $g_3$ are DA in $t$. This follows
  from their expression as modular forms~\cite{WW-62}, and from the fact that
  modular forms satisfy differential equations~\cite{zagier-notes}.
\item We conclude using closure properties of D-algebraic functions. $\Box$
  \end{itemize}
 % \end{proof}

%%%%%%%%%%%%%%%%%%%%%%%%%%%%%%%%%%%%%%%%%%%%%%%%%%%%%%%%%%%%%%%%%%%%%%%%%%%%%%
\section{Further results and final comments}
%%%%%%%%%%%%%%%%%%%%%%%%%%%%%%%%%%%%%%%%%%%%%%%%%%%%%%%%%%%%%%%%%%%%%%%%%%%%%%
Tutte's invariants offer a  new approach  of
quadrant walks, and we are faced with many open problems, mainly related
to the notion of decoupling functions. 
\begin{itemize}
     \item We still need to determine the exact applicability of our
       approach: which models admit a decoupling function? For a
       model with finite group, %our approach proves that 
the existence
       of a decoupling function implies algebraicity. Hence we have
       found all decoupling functions for walks with finite group. But
       what about the 51 non-singular models with an
       infinite group? We have found 9 with a decoupling function and
       (although  not explained  in this abstract) we
       also know   that   36  have
       no decoupling function (this includes for instance the
       model with jumps $\downarrow, \swarrow,\leftarrow
       ,\nearrow$). This leaves 6 more models which might have
       decoupling functions:
\begin{center}
 \begin{tikzpicture}[scale=.5] 
    \draw[->] (0,0) -- (-1,0);
       \draw[->] (0,0) -- (0,1);
    \draw[->] (0,0) -- (1,-1) ;
    \draw[->] (0,0) -- (1,0) ;
        \draw[->] (0,0) -- (0,-1);
  \end{tikzpicture}\hspace{4mm}
 \begin{tikzpicture}[scale=.5] 
    \draw[->] (0,0) -- (-1,0);
        \draw[->] (0,0) -- (0,1) ;
    \draw[->] (0,0) -- (-1,-1) ;
    \draw[->] (0,0) -- (1,0) ;
        \draw[->] (0,0) -- (0,-1);
  \end{tikzpicture}\hspace{4mm} \begin{tikzpicture}[scale=.5] 
        \draw[->] (0,0) -- (-1,1);
    \draw[->] (0,0) -- (0,1);
    \draw[->] (0,0) -- (-1,-1) ;
    \draw[->] (0,0) -- (1,0) ;
    \draw[->] (0,0) -- (1,-1);
       \end{tikzpicture}\hspace{4mm} \begin{tikzpicture}[scale=.5] 
    \draw[->] (0,0) -- (-1,0);
    \draw[->] (0,0) -- (-1,1);
    \draw[->] (0,0) -- (0,1);
    \draw[->] (0,0) -- (-1,-1) ;
    \draw[->] (0,0) -- (1,0) ;
    \draw[->] (0,0) -- (1,-1);
       \end{tikzpicture}\hspace{4mm} \begin{tikzpicture}[scale=.5] 
    \draw[->] (0,0) -- (-1,0);
    \draw[->] (0,0) -- (-1,1);
    \draw[->] (0,0) -- (0,1);
    \draw[->] (0,0) -- (-1,-1) ;
    \draw[->] (0,0) -- (1,0) ;
    \draw[->] (0,0) -- (1,-1);
    \draw[->] (0,0) -- (0,-1);
   \end{tikzpicture}\hspace{4mm} \begin{tikzpicture}[scale=.5] 
        \draw[->] (0,0) -- (0,1);
    \draw[->] (0,0) -- (-1,-1) ;
    \draw[->] (0,0) -- (1,0) ;
    \draw[->] (0,0) -- (1,1);
     \end{tikzpicture}
\end{center}
     \item We have not considered %here
 the 5 singular
       models. But they do admit weak invariants (though not given
       by~\eqref{eq:expression_gluing}) and this raises the question
       of finding decoupling functions for them. 
     \item 
The known results
 on the nature of $Q(x,y)$ can be summarized as follows:
    \begin{center}
      \begin{tabular}{|l|c|c|}
\hline
& Existence of decoupling functions & No decoupling function\\
\hline
Finite group & Algebraic & D-finite 
transcendental\\
\hline
Infinite group  & Differentially algebraic & ?\\
\hline
     \end{tabular}
     \end{center}
    Could it be that  infinite group models without decoupling
    function have a non-differentially algebraic \gf?
     \item 
Can we obtain explicit differential equations in the D-algebraic
cases? One possible approach would be to mimic Tutte's solution
of~\eqref{eq:Tutte}: he first found  a non-linear differential equation
valid for infinitely many values of $q$ (for which $G(1,0)$ is in fact
algebraic), and  then concluded
by a continuity argument.
% that it must hold for any $q$. 
In our
context, this would mean introducing weights so as to obtain a family
of algebraic
models converging to a D-algebraic one.

     \item Another aspect, to be explored in the long version of this
       extended abstract, is the influence of the starting point on
       the nature of the \gf. As examples, Gessel's and Kreweras'
       models starting from $(0,1)$ do not admit decoupling functions,
       which implies their  transcendance. 
\end{itemize}     

\medskip

%\acknowledgements \label{sec:ack}
\noindent {\bf Acknowledgements.} We  thank Irina Kurkova
for interesting discussions. OB is supported by the NSF grant DMS-1400859. KR is partially supported by the
project MADACA from the R\'egion Centre and by the ANR Graal Grant
(ANR-14-CE25-0014).  
We are grateful to one of our FPSAC referees who spotted mistakes in our tables.

%\nocite{*}
%\bibliographystyle{abbrvnat}

% use the following instead if you encounter problems 
%\small
%\normalsize
\bibliographystyle{alpha}
\bibliography{BeBMRa}

\begin{thebibliography}{BMM10}

\bibitem[AS64]{AS}
M.~Abramowitz and I.~Stegun.
\newblock {\em Handbook of mathematical functions with formulas, graphs, and
  mathematical tables}, volume~55 of {\em National Bureau of Standards Applied
  Mathematics Series}.
\newblock U.S. Government Printing Office, Washington, D.C., 1964.

\bibitem[BK10]{BoKa-10}
A.~Bostan and M.~Kauers.
\newblock The complete generating function for {G}essel walks is algebraic.
\newblock {\em Proc. Amer. Math. Soc.}, 138(9):3063--3078, 2010.
\newblock \href{http://arxiv.org/abs/0909.1965}{ArXiv:0909.1965}.

\bibitem[BKR13]{BKR-13}
A.~Bostan, I.~Kurkova, and K.~Raschel.
\newblock A human proof of {G}essel's lattice path conjecture.
\newblock {\em Trans. Amer. Math. Soc. {\rm(to appear)}}, 2013.
\newblock \href{http://arxiv.org/abs/1309.1023}{ArXiv:1309.1023}.

\bibitem[BM02]{bousquet-versailles}
M.~Bousquet-M{\'e}lou.
\newblock Counting walks in the quarter plane.
\newblock In {\em Mathematics and computer science $2$, (Versailles, $2002$)},
  Trends Math., pages 49--67. Birkh\"auser, Basel, 2002.

\bibitem[BM15]{mbm-gessel}
M.~Bousquet-M\'elou.
\newblock An elementary solution of {G}essel's walks in the quadrant.
\newblock \href{http://arxiv.org/abs/1503.08573}{ArXiv:1503.08573}, 2015.

\bibitem[BMJ06]{mbm-jehanne}
M.~Bousquet-M\'elou and A.~Jehanne.
\newblock Polynomial equations with one catalytic variable, algebraic series
  and map enumeration.
\newblock {\em J. Combin. Theory Ser. B}, 96:623--672, 2006.
\newblock \href{http://arxiv.org/abs/math/0504018}{Arxiv:math/0504018}.

\bibitem[BMM10]{BMM-10}
M.~Bousquet-M{\'e}lou and M.~Mishna.
\newblock Walks with small steps in the quarter plane.
\newblock In {\em Algorithmic probability and combinatorics}, volume 520 of
  {\em Contemp. Math.}, pages 1--39. Amer. Math. Soc., Providence, RI, 2010.
\newblock \href{http://arxiv.org/abs/0810.4387}{ArXiv:0810.4387}.

\bibitem[DW15]{denisov-wachtel}
D.~Denisov and V.~Wachtel.
\newblock Random walks in cones.
\newblock {\em Ann. Probab.}, 43(3):992--1044, 2015.
\newblock \href{http://arxiv.org/abs/1110.1254}{Arxiv:1110.1254}.

\bibitem[FIM99]{FIM-99}
G.~Fayolle, R.~Iasnogorodski, and V.~Malyshev.
\newblock {\em Random walks in the quarter-plane}, volume~40 of {\em
  Applications of Mathematics (New York)}.
\newblock Springer-Verlag, Berlin, 1999.

\bibitem[Ges86]{gessel-proba}
I.~M. Gessel.
\newblock A probabilistic method for lattice path enumeration.
\newblock {\em J. Statist. Plann. Inference}, 14(1):49--58, 1986.

\bibitem[GZ92]{gessel-zeilberger}
I.~M. Gessel and D.~Zeilberger.
\newblock Random walk in a {W}eyl chamber.
\newblock {\em Proc. Amer. Math. Soc.}, 115(1):27--31, 1992.

\bibitem[KKZ09]{KaKoZe08}
M.~Kauers, C.~Koutschan, and D.~Zeilberger.
\newblock Proof of {I}ra {G}essel's lattice path conjecture.
\newblock {\em Proc. Nat. Acad. Sci. USA}, 106(28):11502--11505, 2009.
\newblock \href{http://arxiv.org/abs/0806.4300}{ArXiv:0806.4300}.

\bibitem[KR12]{KR-12}
I.~Kurkova and K.~Raschel.
\newblock On the functions counting walks with small steps in the quarter
  plane.
\newblock {\em Publ. Math. Inst. Hautes \'Etudes Sci.}, 116:69--114, 2012.
\newblock \href{http://arxiv.org/abs/1107.2340}{ArXiv:1107.2340}.

\bibitem[KY15]{KaYa-15}
M.~Kauers and R.~Yatchak.
\newblock Walks in the quarter plane with multiple steps.
\newblock In {\em FPSAC 2015}, DMTCS Proceedings, pages 25--36, 2015.
\newblock
  \href{http://fpsac2015.sciencesconf.org/70160}{http://fpsac2015.sciencesconf.org/70160}.

\bibitem[KZ08]{kauers07v}
M.~Kauers and D.~Zeilberger.
\newblock The quasi-holonomic {A}nsatz and restricted lattice walks.
\newblock {\em J. Difference Equ. Appl.}, 14:1119--1126, 2008.
\newblock \href{http://arxiv.org/abs/0806.4318}{ArXiv:0806.4318}.

\bibitem[Lit00]{Li-00}
G.~Litvinchuk.
\newblock {\em Solvability theory of boundary value problems and singular
  integral equations with shift}.
\newblock Kluwer Academic Publishers, Dordrecht, 2000.

\bibitem[MR09]{MiRe09}
M.~Mishna and A.~Rechnitzer.
\newblock Two non-holonomic lattice walks in the quarter plane.
\newblock {\em Theoret. Comput. Sci.}, 410(38-40):3616--3630, 2009.
\newblock \href{http://arxiv.org/abs/math/0701800}{ArXiv:math/0701800}.

\bibitem[Ras12]{Ra-12}
K.~Raschel.
\newblock Counting walks in a quadrant: a unified approach via boundary value
  problems.
\newblock {\em J. Eur. Math. Soc. (JEMS)}, 14(3):749--777, 2012.
\newblock \href{http://arxiv.org/abs/1003.1362}{ArXiv:1003.1362}.

\bibitem[Tut95]{tutte-chromatic-revisited}
W.~T. Tutte.
\newblock Chromatic sums revisited.
\newblock {\em Aequationes Math.}, 50(1-2):95--134, 1995.

\bibitem[WW62]{WW-62}
E.~T. Whittaker and G.~N. Watson.
\newblock {\em A course of modern analysis}.
\newblock Fourth edition. Cambridge University Press, New York, 1962.

\bibitem[Zag91]{zagier-notes}
D.~Zagier.
\newblock Modular forms of one variable.
\newblock Notes based on a course given in Utrecht, Spring 1991. Available at
  \href{http://people.mpim-bonn.mpg.de/zagier/files/tex/UtrechtLectures/UtBook.pdf}{this
  URL}, 1991.

\end{thebibliography}
\label{sec:biblio}
%\normalsize
\appendix

%%%%%%%%%%%%%%%%%%%%%%%%%%%%%%%%%%%%%%%%%%%%%%%%%%%%%%%%%%%%%%%%%%%%%%%
\section{Rational invariants for finite groups}
\label{app:inv}
%%%%%%%%%%%%%%%%%%%%%%%%%%%%%%%%%%%%%%%%%%%%%%%%%%%%%%%%%%%%%%%%%%%%%%%
Recall from~\cite{BMM-10} that there are 23 models with a finite group.  First come 16
models with a vertical symmetry, for which
$
K(x,y)= (1+x^2)\widetilde a(y) +x \widetilde b(y).
$
Hence for a pair $(x,y)$ that cancels the kernel, we have
$
x+\frac{1}{x}= -\frac{\widetilde b(y)}{\widetilde a(y)}
$.
 A possible choice of invariants is thus:
\vskip -3mm $$
I_1(x)=x+\frac{1}{x} , \qquad I_2(y)=-\frac{\widetilde b(y)}{\widetilde a(y)}.
$$
We are left with 7 models. We can restrict the discussion to 3 
since the invariants of two models differing by a symmetry of the square are
related. As these symmetries are generated by the 
reflection in the first diagonal and the reflection in a vertical, we
just need to consider these two cases.

\begin{lem}
\label{lem:invariant_invariants}
  Take  a model $\cS$ with kernel $K(x,y)$ and its diagonal reflection
  $\widetilde \cS$, with kernel
$
\widetilde K(x,y)= K(y,x).
$
Then $\widetilde \cS$ admits invariants if and only if $\cS$ does, and in this case
a possible choice is
$
\widetilde I_1= I_2$  and $\widetilde I_2= I_1.
$
A similar statement holds for the vertical reflection $\overline{\cS}$, with
  kernel
$
\overline{K}(x,y)= x^2 K(\bx,y).
$
A possible choice is in this case
$
\overline{I}_1(x)= I_1(\bx)$ and  $ \overline{I}_2= I_2.
$
\end{lem}
%The proof is elementary.

We can now complete our list of rational invariants:
Table~\ref{tab:ratinv} gives for the 7 remaining models   a
pair $(I_1, I_2)$  satisfying the conditions of
Lemma~\ref{lem:xy-inv}. Finally, Table~\ref{tab:ratinv-weighted} gives invariants for the
four weighted models shown in Table~\ref{tab:decoupling_functions-finite}, right.

\begin{table}[htb]
\small
 \begin{tabular}{|l|c|c|c|c|c|c|}
\hline
&    \begin{tikzpicture}[scale=.4] % q6654
    \draw[->] (0,0) -- (1,-1);
    \draw[->] (0,0) -- (-1,0);
    \draw[->] (0,0) -- (0,1);
  \end{tikzpicture} % Tandem 
&\begin{tikzpicture}[scale=.4] % q6654
    \draw[->] (0,0) -- (1,1);
    \draw[->] (0,0) -- (-1,0);
    \draw[->] (0,0) -- (0,-1);
  \end{tikzpicture}% Kreweras 
&\begin{tikzpicture}[scale=.4] % q6654
    \draw[->] (0,0) -- (-1,-1);
    \draw[->] (0,0) -- (1,0);
    \draw[->] (0,0) -- (0,1);
  \end{tikzpicture}% reverse Kreweras  
&\begin{tikzpicture}[scale=.4] % q6654
    \draw[->] (0,0) -- (1,-1);
    \draw[->] (0,0) -- (-1,1);
    \draw[->] (0,0) -- (-1,0);
    \draw[->] (0,0) -- (1,0);
    \draw[->] (0,0) -- (0,-1);
    \draw[->] (0,0) -- (0,1);
  \end{tikzpicture}% double tandem 
  &\begin{tikzpicture}[scale=.4] % 
    \draw[->] (0,0) -- (-1,-1);
    \draw[->] (0,0) -- (1,1);
    \draw[->] (0,0) -- (-1,0);
    \draw[->] (0,0) -- (1,0);
    \draw[->] (0,0) -- (0,-1);
    \draw[->] (0,0) -- (0,1);
  \end{tikzpicture}% double Kreweras 
   &\begin{tikzpicture}[scale=.4] % 
    \draw[->] (0,0) -- (-1,1);
    \draw[->] (0,0) -- (1,-1);
    \draw[->] (0,0) -- (-1,0);
    \draw[->] (0,0) -- (1,0);
  \end{tikzpicture}\ \ \ % Gouyou
  \begin{tikzpicture}[scale=.4] % 
    \draw[->] (0,0) -- (1,1);
    \draw[->] (0,0) -- (-1,-1);
    \draw[->] (0,0) -- (-1,0);
    \draw[->] (0,0) -- (1,0);
  \end{tikzpicture}% Gessel  
\\
\hline
 $I_1$& $\frac{t}{x^2}-\frac{1}{x}-tx$ & $\frac{t}{x^2}-\frac{1}{x}-tx$ &
 $tx^2-x-\frac{t}{x}$
&$ tx-\frac{t}{x} + \frac{1+2t}{1+\frac{1}{x}}$ &$\frac{t}{x}-tx-  \frac{1+2t}{1+x}$ &$x+\frac{1}{x}-tx^2-\frac{t}{x^2}$
\\ & & & & & &
\\$I_2$& $ty^2-y-\frac{t}{y}$ & $\frac{t}{y^2}-\frac{1}{y}-ty$&$ty^2-y-\frac{t}{y}$&$ \frac{t}{y}-ty -
\frac{1+2t}{1+y}$&$\frac{t}{y}-ty -
\frac{1+2t}{1+y}$  & $y+\frac{1}{y}-ty^2-\frac{t}{y^2}$\\
\hline
  \end{tabular}
  \caption{Rational invariants for models with a finite group: models
    with no vertical symmetry.}
  \label{tab:ratinv}
\end{table}
\normalsize

\begin{table}[htb]
\small
\begin{tabular}[h!]{|c|c|c|c|}
\hline
&  \begin{tikzpicture}[scale=.3] % q6654
    \draw[->] (0,0) -- (-1,0) node[left] {$\scriptstyle 1$};
    \draw[->] (0,0) -- (-1,-1) node[left] {$\scriptstyle 1$};
    \draw[->] (0,0) -- (0,-1) node[below] {$\scriptstyle \lambda$};
    \draw[->] (0,0) -- (1,-1) node[right] {$\scriptstyle 1$};
    \draw[->] (0,0) -- (1,0) node[right] {$\scriptstyle 2$};
    \draw[->] (0,0) -- (1,1) node[right] {$\scriptstyle 1$};
  \end{tikzpicture}&\begin{tikzpicture}[scale=.3] % q6654
    \draw[->] (0,0) -- (-1,0) node[left] {$\scriptstyle 1$};
    \draw[->] (0,0) -- (-1,1) node[left] {$\scriptstyle 1$};
    \draw[->] (0,0) -- (0,1) node[above] {$\scriptstyle 2$};
    \draw[->] (0,0) -- (1,1) node[right] {$\scriptstyle 1$};
    \draw[->] (0,0) -- (1,0) node[right] {$\scriptstyle 2$};
    \draw[->] (0,0) -- (1,-1) node[right] {$\scriptstyle 1$};
    \draw[->] (0,0) -- (0,-1) node[below] {$\scriptstyle 1$};
  \end{tikzpicture} & \begin{tikzpicture}[scale=.3] % q6654
    \draw[->] (0,0) -- (-1,0) node[left] {$\scriptstyle 2$};
    \draw[->] (0,0) -- (-1,1) node[left] {$\scriptstyle 1$};
    \draw[->] (0,0) -- (0,1) node[above] {$\scriptstyle 1$};
    \draw[->] (0,0) -- (-1,-1) node[left] {$\scriptstyle 1$};
    \draw[->] (0,0) -- (1,0) node[right] {$\scriptstyle 1$};
    \draw[->] (0,0) -- (1,-1) node[right] {$\scriptstyle 1$};
    \draw[->] (0,0) -- (0,-1) node[below] {$\scriptstyle 2$};
  \end{tikzpicture}\ \ 
    \begin{tikzpicture}[scale=.3] % q6654
    \draw[->] (0,0) -- (-1,0) node[left] {$\scriptstyle 2$};
    \draw[->] (0,0) -- (-1,1) node[left] {$\scriptstyle 1$};
    \draw[->] (0,0) -- (0,1) node[above] {$\scriptstyle 2$};
    \draw[->] (0,0) -- (1,1) node[right] {$\scriptstyle 1$};
    \draw[->] (0,0) -- (1,0) node[right] {$\scriptstyle 1$};
    \draw[->] (0,0) -- (0,-1) node[below] {$\scriptstyle 1$};
        \draw[->] (0,0) -- (-1,-1) node[left] {$\scriptstyle 1$};
  \end{tikzpicture} 
\\
\hline
$I_1$ &$\frac{t}{x^2}-\frac{1}{x} -x(1+\lambda t)$
&$\frac{{t}^{2}}{x^{2}}- \frac{( 1+2t) t}{x}-{ { ( 3t+1
 ) t}x}-{\frac { ( 1+3t )  ( 4t+1 ) 
}{x+1}}+{\frac { ( 3t+1 ) ^{2}}{ ( x+1 ) ^{2}}}$ &see Lemma \ref{lem:invariant_invariants} 
\\ & & & and the
\\
$I_2$ & $t^2y +\frac{1+\lambda t}{y+1}- \left(\frac{1+\lambda
    t}{y+1}\right)^2$&
$\frac{{t}^{2}}{y^{2}}- \frac{( 1+2t) t}{y}-{ { ( 3t+1
 ) t}y}-{\frac { ( 1+3t )  ( 4t+1 ) 
}{y+1}}+{\frac { ( 3t+1 ) ^{2}}{ ( y+1 ) ^{2}}}$ & previous example
\\
\hline
  \end{tabular}
  \caption{Rational invariants for weighted models.}
  \label{tab:ratinv-weighted}
\end{table}
\normalsize

\end{document}